\documentclass[draft]{amsart}
\usepackage{amssymb,graphicx,enumerate,amsthm}
\usepackage{multirow}
\usepackage{xcolor}

\renewcommand{\S}{\mathcal{S}}
\renewcommand{\L}{\mathcal{L}}

\newcommand{\Z}{{\mathbb{Z}}}

\renewcommand{\i}{\infty}

\newcommand{\beqs}{\begin{equation*}}
\newcommand{\eeqs}{\end{equation*}}
\numberwithin{equation}{section}
\newtheorem{theorem}{Theorem}[section]

\newtheorem{corollary}[theorem]{Corollary}

\begin{document}

\makeatletter
\def\imod#1{\allowbreak\mkern10mu({\operator@font mod}\,\,#1)}
\makeatother

\author{Alexander Berkovich}
   	\address{Department of Mathematics, University of Florida, 358 Little Hall, Gainesville FL 32611, USA}
   	\email{alexb@ufl.edu}

\author{Ali Kemal Uncu}
   \address{Research Institute for Symbolic Computation, Johannes Kepler University, Linz. Altenbergerstrasse 69
A-4040 Linz, Austria}
   \email{akuncu@risc.jku.at}

\thanks{Research of the first author is partly supported by the Simons Foundation, Award ID: 308929. Research of the second author is supported by the Austrian Science Fund FWF, SFB50-07 and SFB50-09 Projects.}

\title[\scalebox{.9}{Polynomial Identities Implying Capparelli's Partition Theorems}]{Polynomial Identities Implying Capparelli's Partition Theorems}
     
\begin{abstract} 
We propose and recursively prove polynomial identities which imply Capparelli's partition theorems.  We also find perfect companions to the results of Andrews, and Alladi, Andrews and Gordon involving $q$-trinomial coefficients. We follow Kur\c{s}ung\"oz's ideas to provide direct combinatorial interpretations of some of our expressions. We make use of the trinomial analogue of Bailey's lemma to derive new identities. These identities relate certain triple sums and products. A couple of new Slater type identities involving bases $q^2$, $q^3$, $q^6$, and $q^{12}$ are also proven. We also discuss a new infinite hierarchy containing these Slater type identities. 
\end{abstract}

  

\date{\today}
   
\dedicatory{To our kind mentor and great inspiration George E. Andrews}
   
\maketitle

\section{Introduction and background}\label{Sec_intro}

A partition $\pi$ is a finite, non-increasing sequence of positive integers $(\pi_1,\pi_2,\dots,\pi_k)$. The $\pi_i$ are called parts of the partition $\pi$, and $\pi_1+\pi_2+\dots+\pi_k$ is called the size of $\pi$. We call $\pi$ \textit{a partition of $n$} if the size of $\pi$ is $n$. Conventionally, we define the empty sequence as the only partition of 0. 

We use the standard notations as in \cite{Theory_of_Partitions} and \cite{Gasper_Rahman}. For formal variables $a_i$ and $q$, and a non-negative integer $N$ 
	\begin{align}
	\nonumber
&(a)_N:=(a;q)_N = \prod_{n=0}^{N-1}(1-aq^n),\text{   and   }  (a;q)_\i:= \lim_{N\rightarrow\infty}(a;q)_N,\\ 
	\nonumber (a_1,&a_2,\dots,a_k;q)_N := (a_1;q)_N(a_2;q)_N\dots (a_k;q)_n\text{ for any}N\in\Z_{\geq0}\cup\{\infty\},
   	\end{align}
	and we define the $q$-binomial coefficients in the classical manner as
%	coefficients, respectively, as
	\begin{align}
 \label{Binom_def}
  \displaystyle {m+n \brack m}_q &:= \left\lbrace \begin{array}{ll}\frac{(q)_{m+n}}{(q)_m(q)_{n}},&\text{for }m, n \geq 0,\\
   0,&\text{otherwise.}\end{array}\right.%\\
\intertext{It is well known that for $m\in\mathbb{Z}_{\geq0}$}
 \label{Binom_limit}
 \displaystyle \lim_{N\rightarrow\infty}{N\brack m}_q &= \frac{1}{(q;q)_m},
 \intertext{for any $j\in \mathbb{Z}_{\geq0}$}
  \label{Binom_limit2} \displaystyle \lim_{M\rightarrow\infty}{2M\brack M+j}_q &= \frac{1}{(q;q)_\infty},
 \intertext{and for $n,\ m\in\mathbb{Z}_{\geq0}$}
\label{Binom_1_over_q}{n+m\brack m}_{q^{-1}} &= q^{-mn} {n+m\brack m}_{q}.
   \end{align}

Let $C_m(n)$ be the number of partitions of $n$ into distinct parts where no part is congruent to $\pm m$ modulo $6$. Define $D_m(n)$ to be the number of partitions of $n$ into parts, not equal to $m$, where the minimal difference between consecutive parts is 2. In fact, the difference between consecutive parts is greater than or equal to $4$ unless consecutive parts are $3k$ and $3k+3$ (yielding a difference of 3), or $3k-1$ and $3k+1$ (yielding a difference of 2) for some $k\in\Z_{>0}$. 
   
In 1988, S.~Capparelli stated two conjectures for $C_m$ and $D_m$ in his thesis \cite{Capparelli_thesis}. The first one was later proven by G.~E.~Andrews \cite{Andrews_Capparelli} in 1992 during the Centenary Conference in Honor of Hans Rademacher. Two years later Lie theoretic proofs were supplied by Tamba and Xie \cite{Tamba} and by Capparelli \cite{Capparelli_proof}. The first of Capparelli's conjectures was stated and proven in the form of Theorem~\ref{capparelliconj}.
\begin{theorem}[Andrews 1992]\label{capparelliconj} For any non-negative integer $n$, \begin{equation}\nonumber
C_1(n) = D_1(n).\end{equation}
\end{theorem}

A year after Capparelli's proof, Alladi, Andrews, and Gordon improved on Theorem~\ref{capparelliconj} in \cite{Refinement}. They gave a refinement of these identities by introducing restrictions on the number of occurrences of parts belonging to certain congruence classes. In particular, they stated and proved the following extension of Theorem~\ref{capparelliconj}.

\begin{theorem}[Capparelli 1994; Alladi, Andrews, Gordon 1995]\label{fulcapparelli} For any non-negative integer $n$ and $m\in\{1,2\}$,\begin{equation}\nonumber
C_m(n) = D_m(n).
\end{equation}
\end{theorem}

In recent paper \cite{Kanade_Russell} Kanade and Russell found the explicit generating functions for the partitions that satisfy the difference conditions of the Capparelli's partition theorem. Namely, for \begin{equation}\label{Qnm}Q(m,n):= 2m^2 + 6mn + 6n^2,\end{equation} we have \begin{align}
\label{KR1_GF}\sum_{\pi\in\mathcal{D}_1}q^{|\pi|}&=\sum_{m,n\geq 0} \frac{q^{Q(m,n)}}{(q;q)_m(q^3;q^3)_n}, \\
\label{KR2_GF}\sum_{\pi\in\mathcal{D}_2}q^{|\pi|}&=\sum_{m,n\geq 0} \frac{q^{Q(m,n)+m}}{(q;q)_m(q^3;q^3)_n} + \sum_{m,n\geq 0} \frac{q^{Q(m,n)+4m+6n+1}}{(q;q)_m(q^3;q^3)_n},  
\intertext{where $\mathcal{D}_m$ (for $m=1$ and 2) be the set of all partitions that satisfy the conditions of $D_m(n)$ for some $n\in\Z_{\geq 0}$. Later we will also be using the notation $\mathcal{D}_{m,N}$ for the partitions from $\mathcal{D}_m$ where all the parts of partitions are $\leq N$.}
\intertext{
Independently, Kur\c{s}ung\"oz \cite{Kagan, Kagan2} discovered the same generating functions' representations with some slight difference in the representation of \eqref{KR2_GF}:}
\label{Kagan_GF} \sum_{\pi\in\mathcal{D}_2}q^{|\pi|}&=\sum_{m,n\geq 0} \frac{q^{Q(m,n)+m+3n}}{(q;q)_m(q^3;q^3)_n} + \sum_{m,n\geq 0} \frac{q^{Q(m,n)+3m+6n+1}}{(q;q)_m(q^3;q^3)_n}.
\end{align} 
With elementary manipulations one can easily show that these two representations \eqref{KR2_GF}--\eqref{Kagan_GF} are equivalent.

Provided that the Capparelli Partition Theorem is valid, these imply \begin{align}
\label{Cap1}\sum_{m,n\geq 0} \frac{q^{Q(m,n)}}{(q;q)_m(q^3;q^3)_n} &= (-q^2,-q^4;q^6)_\infty (-q^3;q^3)_\infty,\\
\label{Cap2}\sum_{m,n\geq 0} \frac{q^{Q(m,n)+m+3n}}{(q;q)_m(q^3;q^3)_n} + \sum_{m,n\geq 0}& \frac{q^{Q(m,n)+3m+6n+1}}{(q;q)_m(q^3;q^3)_n} = (-q,-q^5;q^6)_\infty(-q^3;q^3)_\infty,\\
\label{Cap3}\sum_{m,n\geq 0} \frac{q^{Q(m,n)+m}}{(q;q)_m(q^3;q^3)_n} + \sum_{m,n\geq 0}& \frac{q^{Q(m,n)+4m+6n+1}}{(q;q)_m(q^3;q^3)_n} = (-q,-q^5;q^6)_\infty(-q^3;q^3)_\infty.
\end{align}
We would like to remark that Sills \cite{Sills} discovered different series representations of the product in \eqref{Cap1}.

The goal of this paper is to find polynomial extensions of three identities \eqref{Cap1}-\eqref{Cap3}, and prove these polynomial identities using recurrences. In particular, we prove the following theorem.

%%%%%
%%%RECALL THAT THIS IS G_1, 3N-2
%%%%
\begin{theorem}\label{Intro_Main_THM}For any $N\in\Z_{\geq 0}$, we have
\begin{align} 
\label{Fin_Cap_1_1}&\sum_{m,n\geq 0} q^{Q(m,n)}{3(N-2n-m)\brack m}_q {2(N-2n-m)+n\brack n}_{q^3}= \sum_{l=0}^{N} q^{3{N-2l \choose 2}} {N\brack 2l}_{q^3} (-q^2,-q^4;q^6)_l,\\\nonumber\\
\nonumber &\sum_{m,n\geq 0} q^{Q(m,n)+m+3n}{3(N-2n-m)+2\brack m}_q {2(N-2n-m)+n+1\brack n}_{q^3}\\ 
\label{Fin_Cap_2_1}&\hspace{1cm}+ \sum_{m,n\geq 0} q^{Q(m,n)+3m+6n+1}{3(N-2n-m)\brack m}_q {2(N-2n-m)+n\brack n}_{q^3} \\
\nonumber&\hspace{.1cm}= \sum_{l=0}^{N} q^{3{N-2l \choose 2}} {N+1\brack 2l+1}_{q^3} (-q;q^6)_{l+1}(-q^5;q^6)_l,\\
\nonumber\\
\nonumber &\sum_{m,n\geq 0} q^{Q(m,n)+m}{3(N-2n-m)+2\brack m}_q {2(N-2n-m)+n+2\brack n}_{q^3}\\ \label{Fin_Cap_3_1}&\hspace{1cm}+ \sum_{m,n\geq 0} q^{Q(m,n)+4m+6n+1}{3(N-2n-m)-1\brack m}_q {2(N-2n-m)+n\brack n}_{q^3} \\
\nonumber&\hspace{1cm}+ q^{Q(0,N/2) + 3N + 1}\chi(N)\\
\nonumber&\hspace{.1cm}= \sum_{l=0}^{N} q^{3{N-2l \choose 2}} {N+1\brack 2l+1}_{q^3} (-q;q^6)_{l+1}(-q^5;q^6)_l,
\end{align}
where $\chi(N)$ is 1 if $N$ is even, and 0 otherwise.
\end{theorem}

In Section~\ref{Sec_MultiSums}, after proving Theorem~\ref{Intro_Main_THM}, we will show that in the limit $N\rightarrow \infty$ the identities \eqref{Fin_Cap_1_1}-\eqref{Fin_Cap_3_1} turn into the identities \eqref{Cap1}-\eqref{Cap3}, respectively.

The rest of this paper is organized as follows: In Section~\ref{Sec_recurrences}, we will discuss various polynomial representations of some generating functions related to $\mathcal{D}_{m,N}$ following developments in Alladi, Andrews, and Gordon's paper \cite{Refinement} and the authors' paper \cite{BerkovichUncu1}. Section~\ref{Sec_MultiSums} has the proof of Theorem~\ref{Intro_Main_THM} and other polynomial identities related to the Capparelli partition theorem. Combinatorial insights into the double sum generating functions that appear in this work will be given through Kur\c{s}ung\"oz style rules of motion in Section~\ref{Sec_Motions}. In particular we will show that \[\sum_{m,n\geq 0} q^{Q(m,n)}{3(N-2n-m+1)\brack m}_q {2(N-2n-m+1)+n\brack n}_{q^3}=\sum_{\pi\in\mathcal{D}_{1,3N+1}} q^{|\pi|},\] for any integer $N\geq 1$. In Section~\ref{Sec_Trinomials}, we present analytic and combinatorial results in connection to the $q$-trinomial coefficients, which are perfect companions to the earlier works of Andrews \cite{Andrews_Capparelli} and Alladi, Andrews and Gordon \cite{Refinement}. Section~\ref{Sec_Futher} is reserved for some highly intriguing $q$-series implications of this study. Possible future work and more on polynomial identities that imply Capparelli's identities are briefly mentioned in Section~\ref{Sec_Outlook}.

\section{Alladi, Andrews, Gordon polynomials and their variants}\label{Sec_recurrences}

For $m\in\{1,2\}$, let $G_{m,N}:=G_{m,N} (a,b,q)$ be the generating function for number of partitions where \begin{enumerate}[\itshape i.] 
\item parts are not equal to $m$, 
\item the difference between consecutive parts is greater or equal than $4$ unless either consecutive parts are both consecutive multiples of $3$ or add to a multiple of 6, 
\item largest part is less than or equal to $N$, \item the exponent of $a$ counts the number of parts congruent to $2$ modulo $3$, and 
\item the exponent of $b$ counts the number of parts congruent to $1$ modulo $3$.
\end{enumerate} 

Let $\nu_{i,M}(\pi)$ be the number of $i$ modulo $M$ parts in the partition $\pi$. Then, for $m\in\{1,2\}$, the above notation can be written as \[G_{m,N}:=G_{m,N} (a,b,q) := \sum_{\pi\in \mathcal{D}_{m,N}} a^{\nu_{2,3}(\pi)}b^{\nu_{1,3}(\pi)}q^{|\pi|},\] where $\mathcal{D}_{m,N}$ is defined as in the Section~\ref{Sec_intro}.

For a positive integer $N$, it is easy to see that these generating functions satisfy the recursion relations:
\begin{align}
\label{AAGrec3}G_{m,3N-1} &= G_{m,3(N-1)+1} + aq^{3(N-1)+2}G_{3(N-2)+1},\\
\label{AAGrec2}G_{m,3N} &= G_{m,3(N-1)+2} + q^{3N}G_{m,3(N-1)},\\
\label{AAGrec1}G_{m,3N+1} &= G_{m,3N} + bq^{3N+1}G_{m,3(N-1)} + abq^{6N}G_{m,3(N-2)+1}.
\end{align}
Moreover, the first four initial conditions 
\begin{equation}%\label{initial}
\nonumber\begin{array}{cccl} G_{m,-2}=\delta_{1,m}, & G_{m,-1}=1, & G_{m,0}=1, &\text{and } G_{m,1} = 1 + \delta_{2,m}bq,
\end{array} \end{equation}
where $\delta_{i,j}$ is the Kronecker delta function, define both generating function sequences uniquely.

In the original proof of Theorem~\ref{capparelliconj}, by combining the recurrences \eqref{AAGrec3}-\eqref{AAGrec1}, Andrews finds a third order recurrence for the $G_{1,3N+1}(1/t,t,q)$. Later in \cite{Refinement} this recurrence is refined and stated as \begin{align}
\nonumber G_{m,3N+1}(a,b,q) &= (1+q^{3N})G_{m,3(N-1)+1}\\
 \label{Andrews_recurrence}&+  (aq^{3N-1}+bq^{3N+1}+abq^{6N})G_{m,3(N-2)+1}\\ \nonumber
 &+ abq^{6N-3}(1-q^{3N-3})G_{m,3(N-3)+1}.
 \end{align}
 The recurrence \eqref{Andrews_recurrence} with the initial conditions 
 \begin{align}  \label{Initial_conditions_Min} G_{m,-2}&=\delta_{1,m}, \text{ and }\ G_{m,1} = 1 + \delta_{2,m}bq,
 \end{align} uniquely defines the sequence of the generating functions $G_{m,3N+1}$, for any non-negative $N$, and for $m\in\{1,2\}$.
 
 One can iterate this recurrence \eqref{Andrews_recurrence} once to get a recurrence of order 4. This is done by applying the recurrence once more for the term $q^{3N}G_{m,3(N-1)+1}$.
 \begin{align}
 \nonumber G_{m,3N+1}(a,b,q) &= G_{m,3(N-1)+1}\\
 \label{Andrews_iterated_recurrence}&+  (q^{3N}+q^{6N-3}+aq^{3N-1}+bq^{3N+1}+abq^{6N})G_{m,3(N-2)+1}\\ \nonumber
 \nonumber &+ (aq^{6N-4}+bq^{6N-2}+abq^{6N-3})G_{m,3(N-3)+1} \\
 \nonumber &+ abq^{9N-9}(1-q^{3N-6})G_{m,3(N-4)+1}.
 \end{align}
This recurrence together with the initial conditions of \eqref{Initial_conditions_Min} and \begin{align}
\label{Initial_condition_fourth}G_{m,4}\, &= 1 + q^3 + bq^4 + \delta_{1,m}(aq^2+abq^6) + \delta_{2,m}bq,
\end{align} defines the sequence of $G_{m,3N+1}$ for any non-negative $N$, and for $m\in\{1,2\}$.
The recurrence \eqref{Andrews_iterated_recurrence} will be used later. 
 
Analogous to \eqref{Andrews_recurrence}, we combine the same recurrences \eqref{AAGrec3}-\eqref{AAGrec1} to get the third order recurrence relation for the $G_{m,3N}$ functions. We have \begin{align}
\nonumber G_{m,3N}(a,b,q) &= (1+q^{3N})G_{m,3(N-1)}\\
\label{New_recurrence} &+  (aq^{3N-1}+bq^{3N-2}+abq^{6N-6})G_{m,3(N-2)}\\
\nonumber &+ abq^{6N-6}(1-q^{3N-6})G_{m,3(N-3)}.
 \end{align}
This can be done by writing \eqref{AAGrec3} in \eqref{AAGrec2} and solving the outcome recurrrence with \eqref{AAGrec1} together as a linear system. The recurrence \eqref{New_recurrence} with the initial conditions \begin{equation}\label{Init_cond_0mod3} \ G_{m,0} = 1,\text{ and }G_{m,3} = 1+q^3+\delta_{m,1}aq^2+\delta_{m,2}bq \end{equation} uniquely define this sequence of $G_{m,3N}$ polynomials for any non-negative $N$, and for $m\in\{1,2\}$.

One can show that for $m=1,2$, $G_{m,3N}$ and $G_{m,3N-2}$ have explicit polynomial representations. Hence, \eqref{AAGrec3} provides a polynomial representation for $G_{m,3N-1}$. In \cite{Refinement}, Alladi, Andrews and Gordon found the polynomial representation for the generating function $G_{1,3N+1}(a,b,q)$.
\begin{theorem}[Alladi, Andrews, Gordon, 1995]
\begin{equation}\label{polyAAG} G_{1,3N+1}(a,b,q) = \sum_{l=0}^{\lfloor (N+1)/2 \rfloor} q^{3{N-2l+1 \choose  2}} \genfrac{[}{]}{0pt}{}{N+1}{2l}_{q^3}(-aq^2,-bq^4;q^6)_l,\end{equation}where $\lfloor x\rfloor$ denotes the greatest integer $\leq x$.
\end{theorem}
In fact, we discovered new formulas for all three remaining generating functions. In particular, let
\begin{equation}\label{U_poly} U(a,b,q,N) :=\sum_{l=0}^{\lfloor N/2 \rfloor} q^{3{N-2l \choose  2}} \genfrac{[}{]}{0pt}{}{N+1}{2l+1}_{q^3}(-aq^5;q^6)_l(-bq;q^6)_{l+1}.\end{equation} The $q$-Zeilberger algorithm implemented by Paule and Riese \cite{qZeil} automatically proves that $U$ satisfies the recurrence \eqref{Andrews_recurrence}. We prove that $ G_{2,3N+1}(a,b,q) = U(a,b,q,N),$ by noting that \[U(a,b,q,-1)= 0,\text{ and }U(a,b,q,0)= 1+bq,\] initial conditions uniquely define the sequence of $U(a,b,q,N)$ for any non-negative $N$ and by comparing these initial conditions with \eqref{Initial_conditions_Min}. Hence, 
\begin{theorem}
\begin{equation}\label{G2_U_poly}
 G_{2,3N+1}(a,b,q) =\sum_{l=0}^{\lfloor N/2 \rfloor} q^{3{N-2l \choose  2}} \genfrac{[}{]}{0pt}{}{N+1}{2l+1}_{q^3}(-aq^5;q^6)_l(-bq;q^6)_{l+1}.
\end{equation}
\end{theorem}

Moreover, we define the polynomials 
\begin{align}
\label{S_poly}S_N:=S(a,b,q,N) &:= \sum_{l=0}^{\lfloor N/2 \rfloor} q^{3{N-2l \choose  2}} \genfrac{[}{]}{0pt}{}{N+1}{2l+1}_{q^3}(-aq^2,-bq^4;q^6)_{l},\\
\label{T_poly}T_N:=T(a,b,q,N) &:= \sum_{l=0}^{\lfloor N/2 \rfloor} q^{3{N-2l \choose  2}} \genfrac{[}{]}{0pt}{}{N+1}{2l+1}_{q^3}(-aq^5,-bq;q^6)_{l}.
\end{align}
The $q$-Zeilberger algorithm proves that these polynomials satisfy the recurrences \vspace{.1cm}
\begin{align}
\label{S_rec}S_N &= (1+q^{3N})S_{N-1} + q^{3N-6}(bq^4+aq^2+abq^{3N})S_{N-2}+abq^{6N-9}(1-q^{3(N-1)})S_{N-3},\\[-1,5ex]\nonumber\\
\label{T_rec}T_N &= (1+q^{3N})T_{N-1} + q^{3N-6}(bq+aq^5+abq^{3N})T_{N-2}+abq^{6N-9}(1-q^{3(N-1)})T_{N-3},
\end{align} respectively.
We prove the following. 
\begin{theorem}\label{THM_0mod3}
\begin{align}
\label{G_0mod3_S}G_{1,3N}(a,b,q) &= S(a,b,q,N) + aq^{3N-1} S(a,b,q,N-1),\\[-1,5ex]\nonumber\\
\label{G_0mod3_T}G_{2,3N}(a,b,q) &= T(a,b,q,N) + bq^{3N-2} T(a,b,q,N-1).
\end{align} 
\end{theorem}
To prove Theorem~\ref{THM_0mod3} one needs to show that the right-hand sides of \eqref{G_0mod3_S} and \eqref{G_0mod3_T} both satisfy the recurrence \eqref{New_recurrence}. To this end, one needs to combine the recurrence \eqref{S_rec} (and \eqref{T_rec}) with a shifted one of the same recurrence, and manually write $S_N$ (and $T_N$) as an expression in $X_N= S_N + aq^{3N-1} S_{N-1}$ (and $Y_N= T_N + bq^{3N-2} T_{N-1}$, respectively). This way one sees that the third order recurrences for $X_N$ and $Y_N$ are the same as \eqref{New_recurrence}. Finally, one compares the initial conditions 
\[\begin{array}{ll}
%S(a,b,q,-1) + aq^{-4} S(a,b,q,-2) = 0, &T(a,b,q,-1) + bq^{-5} T(a,b,q,-2) = 0, \\
X_0 = S(a,b,q,0) + aq^{-1} S(a,b,q,-1) = 1, &Y_0=T(a,b,q,0) + bq^{-2} T(a,b,q,-1) = 1, \\[-1.5ex]\\
X_1 = S(a,b,q,1) + aq^{2} S(a,b,q,0) = 1+aq^2+q^3,&Y_1= T(a,b,q,1) + bq T(a,b,q,0) = 1+bq+q^3,
\end{array}\] which uniquely determine $X_N$ and $Y_N$ for any non-negative $N$, and for $m\in\{1,2\}$,
with the initial conditions in \eqref{Init_cond_0mod3} to finish the proof of Theorem~\ref{THM_0mod3}.

\section{Recursive proof of Theorem~\ref{Intro_Main_THM} and related results}\label{Sec_MultiSums}

In order to prove the identities of Theorem~\ref{Intro_Main_THM}, we need to show that both sides of the identities satisfy the same recurrences with the same initial conditions. In Section~\ref{Sec_recurrences} we have already shown general recurrences \eqref{Andrews_recurrence} (and \eqref{Andrews_iterated_recurrence}) for the right-hand side sums of the identities \eqref{Fin_Cap_1_1}-\eqref{Fin_Cap_3_1} with extra variables $a$ and $b$. In this section, we will be using these recurrences with $a=b=1$.

For the left-hand sides of the identities \eqref{Fin_Cap_1_1}-\eqref{Fin_Cap_3_1} we employ the Mathematica package qMultiSum developed by Riese \cite{qMultiSum}. This package finds and proves recurrences for multi-sums under the assumption that the recurrences are always homogeneous.

We start by proving the identity \eqref{Fin_Cap_1_1}. The left-hand side summand \begin{equation}\label{Summand_Fin_1_1}F_{N,m,n}(q):=q^{Q(m,n)}{3(N-2n-m)\brack m}_q {2(N-2n-m)+n\brack n}_{q^3},\end{equation} where $Q(m,n)$ is as in \eqref{Qnm}, satisfies the recurrence\vspace{.1cm} \begin{align}
\nonumber F_{N,m,n}(q) &= F_{N-1,m,n}(q)  + q^{6N-9}(1+q^3) F_{N-2,m,n-1}(q) + q^{3N-4}(1+q+q^2)F_{N-2,m-1,n}(q)\\[-1.5ex]\nonumber\\
\label{Summand_rec1}&+q^{6N+10}(1+q+q^2)F_{N-3,m-2,n}(q) - q^{12N-27}F_{N-4,m,n-2}(q)+ q^{9N-18}F_{N-4,m-3,n}(q),
\end{align} for any $N\geq 4$ and $m,n \in \Z$. Summing this recurrence with respect to $m$ and $n$ over $\Z$ and recalling the $q$-binomial coefficients vanish when $m$ or $n$ is negative \eqref{Binom_def}. For $N\geq 4$ we have 
\begin{align}\nonumber\L_N(q) &= \L_{N-1}(q) + q^{3N-4}(1+q+q^2+q^{3N-2}+q^{3N-5})\L_{N-2}(q)\\[-1.5ex]\nonumber\\
\label{Sum_rec1}&+ q^{6N-10}(1+q+q^2)\L_{N-3}(q) + q^{9N-18}(1-q^{3N-9})\L_{N-4}(q),
\end{align} where \[\L_N(q)= \sum_{m,n\geq 0} F_{N,m,n}(q)\] is the left-hand side of \eqref{Fin_Cap_1_1}. This matches perfectly with the recurrence \eqref{Andrews_iterated_recurrence} with $N\mapsto N-1$ and $a=b=1$. We need to check 4 initial conditions for $\L_N(q)$: \begin{align*}
\L_0(q) &= \L_1(q) = 1,\\ 
\L_2(q) &=1+{q}^{4}+{q}^{3}+{q}^{2}+{q}^{6},\\
\L_3(q) &= 1+q^2+q^3+q^4+q^5+2q^6+q^7+q^8+2q^9+q^{10}+q^{12}.
\end{align*} These initial conditions match the values of the $G_{1,-2}(1,1,q)$, $G_{1,1}(1,1,q)$, $G_{1,4}(1,1,q)$ and $G_{1,7}(1,1,q)$, respectively. The recurrence \eqref{Sum_rec1}, together with these initial conditions, shows that the $\L_N$ can now be stated as \[\sum_{m,n\geq 0} q^{Q(m,n)}{3(N-2n-m)\brack m}_q {2(N-2n-m)+n\brack n}_{q^3} = G_{1,3N-2}(1,1,q).\] This together with \eqref{polyAAG} proves the identity \eqref{Fin_Cap_1_1}.

We go through the same argument for the identities \eqref{Fin_Cap_2_1} and \eqref{Fin_Cap_3_1}. Let ${}_i\hat{F}_{N,m,n}(q)$ for $i \in \{1,2,3,4\}$ represent the first and second summands of the double sums of \eqref{Fin_Cap_2_1} and first and second summands of the double sums of \eqref{Fin_Cap_3_1} in this order. One observes that for $i=1,$ and 2, ${}_i\hat{F}_{N,m,n}(q)$ satisfy the recurrence\vspace{.1cm} \begin{align}
\nonumber {}_i\hat{F}_{N,m,n}(q) &= {}_i\hat{F}_{N-1,m,n}(q) + q^{6N-3}(1+q^3) {}_i\hat{F}_{N-2,m,n-1}(q) + q^{3N-1}(1+q+q^2){}_i\hat{F}_{N-2,m-1,n}(q)\\[-1.5ex]\nonumber\\
\label{Summand_rec2_3}&+q^{6N-4}(1+q+q^2){}_i\hat{F}_{N-3,m-2,n}(q) - q^{12N-15}{}_i\hat{F}_{N-4,m,n-2}(q)+ q^{9N-9}{}_i\hat{F}_{N-4,m-3,n}(q),
\end{align} for any $N\geq 4$ and $m,n \in \Z$. Summands  \begin{align*}
{}_3{F}_{N,m,n}(q) &= q^{Q(m,n)+m}{3(N-2n-m)+2\brack m}_q {2(N-2n-m)+n\brack n+2}_{q^3},\\
{}_4{F}_{N,m,n}(q) &= q^{Q(m,n)+4m+6n+1}{3(N-2n-m)-1\brack m}_q {2(N-2n-m)+n\brack n}_{q^3}
\end{align*} also satisfy this recurrence wth non-homogeneous corrections. This slight inconvenience is not directly seen from the computer implementation due to the homogeneity assumption of the qMultiSum package. The source of the non-homogeneous terms is the nonuniversality of the original $q$-binomial recurrences: \begin{equation}\label{q_bin_rec}
{n+m\brack m}_q = {n+m-1\brack m}_q + q^{n} {n+m-1\brack m-1}_q.
\end{equation}As clear from the following example, this recurrence fails at $m=n=0$: \[1 = {0\brack 0 }_q \not= {-1\brack0}_q + {-1\brack-1}_q=0+0. \] This gives rise to the non-homogeneous terms in the recurrences. The functions ${}_3\hat{F}_{N,m,n}(q)$ and ${}_4\hat{F}_{N,m,n}(q)$ satisfy the recurrences\vspace{.1cm}
\begin{align}
\nonumber {}_3\hat{F}_{N,m,n}(q) &= {}_3\hat{F}_{N-1,m,n}(q) + q^{6N-3}(1+q^3) {}_3\hat{F}_{N-2,m,n-1}(q) + q^{3N-1}(1+q+q^2){}_3\hat{F}_{N-2,m-1,n}(q)\\[-1.5ex]\nonumber\\
\label{Summand_rec3}&+q^{6N-4}(1+q+q^2){}_3\hat{F}_{N-3,m-2,n}(q) - q^{12N-15}{}_3\hat{F}_{N-4,m,n-2}(q)+ q^{9N-9}{}_3\hat{F}_{N-4,m-3,n}(q)\\[-1.5ex]\nonumber\\\nonumber  &\hspace{-1cm}+ \chi(N)\left(\delta_{m,0}\delta_{n,N/2}q^{\frac{3}{2}N^2} + \delta_{m,2}\delta_{N/2-1}q^{\frac{3}{2}N^2+4} \right)+ \chi(N+1)\delta_{m,1}\delta_{n,(N-1)/2}(1+q)q^{\frac{3N^2 +1}{2}},\\
\intertext{and}
\nonumber {}_4\hat{F}_{N,m,n}(q) &= {}_4\hat{F}_{N-1,m,n}(q) + q^{6N-3}(1+q^3) {}_4\hat{F}_{N-2,m,n-1}(q) + q^{3N-1}(1+q+q^2){}_4\hat{F}_{N-2,m-1,n}(q)\\[-1.5ex]\nonumber\\
\label{Summand_rec4}&+q^{6N-4}(1+q+q^2){}_4\hat{F}_{N-3,m-2,n}(q) - q^{12N-15}{}_4\hat{F}_{N-4,m,n-2}(q)+ q^{9N-9}{}_4\hat{F}_{N-4,m-3,n}(q)\\[-1.5ex]\nonumber\\\nonumber  &\hspace{-1cm}+ \chi(N)\delta_{m,1}\delta_{n,N/2-1}(1+q)q^{\frac{3}{2}N^2+1} + \chi(N+1)\left(\delta_{m,0}\delta_{n,(N-1)/2}q^{\frac{3N^2 -1}{2}}+\delta_{m,2}\delta_{n,(N-3)/2} q^{\frac{3N^2-1}{2}+4}\right),
\end{align} for any $N\geq 4$, and for any $m,n\in\Z$, where $\chi(N)$ is as its defined in Theorem~\ref{Intro_Main_THM}.

Let \[{}_i\S_{N}(q) = \sum_{m,n\geq 0} {}_i\hat{F}_{N,m,n}(q),\] for $i \in\{1,2,3,4\}$. From the recurrence \eqref{Summand_rec2_3} and \eqref{Summand_rec4} one sees that each ${}_i\S_{N}(q)$ satisfy the recurrences
\begin{align}\nonumber {}_i\S_{N}(q)  &= {}_i\S_{N-1}(q) + q^{3N-1}(1+q+q^2+q^{3N-2}+q^{3N+1}){}_i\S_{N-2}(q)\\[-1.5ex]\nonumber\\\label{Sum_rec_1_2_3}
&+q^{6N-4}(1+q+q^2){}_i\S_{N-3}(q) +q^{9N-9}(1-q^{3N-6}){}_i\S_{N-4}(q),\\\intertext{for $i=1$ and $2$, and the non-homogeneous recurrences}
\nonumber{}_3\S_{N}(q)  &= {}_3\S_{N-1}(q) + q^{3N-1}(1+q+q^2+q^{3N-2}+q^{3N+1}){}_3\S_{N-2}(q)\\[-1.5ex]\nonumber\\
\label{Sum_rec_2_4}&+q^{6N-4}(1+q+q^2){}_3\S_{N-3}(q) +q^{9N-9}(1-q^{3N-6}){}_3\S_{N-4}(q)\\[-1.5ex]\nonumber\\
\nonumber&+\chi(N)(1+q^4)q^{\frac{3}{2} N^2 } + \chi(N+1)(1+q)q^\frac{3N^2 +3}{2},\\\intertext{and}
\nonumber{}_4\S_{N}(q)  &= {}_4\S_{N-1}(q) + q^{3N-1}(1+q+q^2+q^{3N-2}+q^{3N+1}){}_4\S_{N-2}(q)\\[-1.5ex]\nonumber\\
\label{Sum_rec_2_4}&+q^{6N-4}(1+q+q^2){}_4\S_{N-3}(q) +q^{9N-9}(1-q^{3N-6}){}_4\S_{N-4}(q)\\[-1.5ex]\nonumber\\
\nonumber&+\chi(N)(1+q)q^{\frac{3}{2} N^2 + 1} + \chi(N+1)(1+q^4)q^\frac{3N^2 -1}{2}.
\end{align} Note that \eqref{Sum_rec_1_2_3} is the same recurrence as \eqref{Andrews_iterated_recurrence} with $a=b=1$. Let \[{}_1\hat{\L}_N(q) = {}_1\S_{N}(q)+{}_2\S_{N}(q),\text{ and }{}_2\hat{\L}_N(q) = {}_3\S_{N}(q)+{}_4\S_{N}(q)+\chi(N)q^{Q(0,N/2)+3N+1}.\] For any $N\ge 4,$ polynomials ${}_1\L_N(q)$ and ${}_2\hat{\L}_N(q)$ both satisfy the recurrence
\begin{align}\nonumber {}_i\hat{\L}_{N}(q)  &= {}_i\hat{\L}_{N-1}(q) + q^{3N-1}(1+q+q^2+q^{3N-2}+q^{3N+1}){}_i\hat{\L}_{N-2}(q)\\[-1.5ex]\nonumber\\\label{Sum_rec_Li}
&+q^{6N-4}(1+q+q^2){}_i\hat{\L}_{N-3}(q) +q^{9N-9}(1-q^{3N-6}){}_i\hat{\L}_{N-4}(q).
\end{align} Note that the recurrence \eqref{Sum_rec_Li} is the same as \eqref{Andrews_iterated_recurrence}, where $a=b=1$. Comparing the initial conditions
\begin{equation}
{}_i\hat{\L}_{-1}(q) = 0, \hspace{.5cm} {}_i\hat{\L}_0(q) = 1+q,\text{ and }{}_i\hat{\L}_1(q) = 1+q+q^3+q^4,
\end{equation} together with the recurrences \eqref{Sum_rec_Li} and \eqref{Andrews_iterated_recurrence}, and with the initial conditions \eqref{Initial_conditions_Min}, \eqref{Initial_condition_fourth} with $m=2$ and $a=b=1$, shows that the left-hand sides of \eqref{Fin_Cap_2_1} and \eqref{Fin_Cap_3_1} are both equal to $G_{2,3N+1}$.

Using $q$-Gauss sum \cite[p.236, II.8]{Gasper_Rahman} and Euler's Partition Theorem \cite{Theory_of_Partitions} we see that \begin{equation}\label{odd_or_even_identity_which_doesnt_matter}
\sum_{\substack{t\geq 0\\ t\equiv a\text{ $($mod }2)}} \frac{q^{t \choose 2}}{(q;q)_t}  = \frac{1}{(q;q^2)_\infty} = (-q;q)_\infty
\end{equation} for $a \in \{0,1\}$.

The limit \eqref{Binom_limit} with the identity \eqref{odd_or_even_identity_which_doesnt_matter} is enough to show that as $N$ tends to infinity the identities \eqref{Fin_Cap_1_1}-\eqref{Fin_Cap_3_1} turn into \eqref{Cap1}-\eqref{Cap3}, respectively.

In the spirit of \eqref{Fin_Cap_1_1}, the following two generating function interpretations and the following analytic identities are true.

\begin{theorem}\label{THM_C1_23} For $Q(m,n):= 2m^2 + 6mn + 6n^2$
\begin{align}
\nonumber\sum_{\pi\in\mathcal{D}_{1,3N}} q^{|\pi|} &= \sum_{m,n\geq 0} q^{Q(m,n)} {3(N-2n-m)+2 \brack m }_q {2(N-2n-m)+n+1\brack n}_{q^3}\\
&= \sum_{l=0}^N q^{3{N-2l \choose 2}} {N+1 \brack 2l+1}_{q^3}(-q^2,-q^4;q^6)_l + q^{3N-1}\sum_{l=0}^N {q^{3{N-2l-1\choose 2}}}{N \brack 2l+1}_{q^3}(-q^2,-q^4;q^6)_l,\\[-1.5ex]\nonumber\\
\nonumber \sum_{\pi\in\mathcal{D}_{1,3N+2}} q^{|\pi|} & = \sum_{m,n\geq0 } q^{Q(m,n)} {3(N-2n-m) \brack m }_q {2(N-2n-m)+n\brack n}_{q^3}\\
&= \sum_{l=0}^N q^{3{N-2l \choose 2}} {N \brack 2l}_{q^3}(-q^2,-q^4;q^6)_l + q^{3N-1}\sum_{l=0}^N{q^{3{N-1\choose 2l}}}(-q^2,-q^4;q^6)_l.
\end{align}
\end{theorem}

In the spirit of \eqref{Fin_Cap_2_1}, the following two generating function interpretations and the following analytic identities are true.
\begin{theorem}\label{THM_C2_23} For $Q(m,n):= 2m^2 + 6mn + 6n^2$
\begin{align}
\nonumber\sum_{\pi\in\mathcal{D}_{2,3N}} q^{|\pi|} &= \sum_{m,n\geq 0} q^{Q(m,n)} {3(N-2n-m+1) \brack m }_q {2(N-2n-m)+n\brack n}_{q^3}\\
&+\sum_{m,n\geq 0} q^{Q(m,n)+3m+6n+1} {3(N-2n-m)-1 \brack m }_q {2(N-2n-m)+n-1\brack n}_{q^3}\\
\nonumber&= \sum_{l=0}^N q^{3{N-2l \choose 2}} {N+1 \brack 2l+1}_{q^3}(-q,-q^5;q^6)_l 
+ q^{3N-2}\sum_{l=0}^N {q^{3{N-2l-1\choose 2}}}{N \brack 2l+1}_{q^3}(-q,-q^5;q^6)_l,\\[-1.5ex]\nonumber\\
\nonumber \sum_{\pi\in\mathcal{D}_{2,3N+2}} q^{|\pi|} & = 
\nonumber\sum_{m,n\geq 0} q^{Q(m,n)+m+3n} {3(N-2n-m+1) \brack m }_q {2(N-2n-m)+n-1\brack n}_{q^3}\\
&+\sum_{m,n\geq 0} q^{Q(m,n)+3m+6n+1} {3(N-2n-m)+1 \brack m }_q {2(N-2n-m)+n\brack n}_{q^3}\\
\nonumber&= \sum_{l=0}^N q^{3{N-2l \choose 2}} {N+1 \brack 2l+1}_{q^3}(-q;q^6)_{l+1}(-q^5;q^6)_l\\
\nonumber&+ q^{3N+2}\sum_{l=0}^N {q^{3{N-2l-1\choose 2}}}{N \brack 2l+1}_{q^3}(-q;q^6)_{l+1}(-q^5;q^6)_l.
\end{align}
\end{theorem}

In the spirit of \eqref{Fin_Cap_3_1}, the following two generating function interpretations and the following analytic identities are true.
\begin{theorem}\label{THM_C3_23} For $Q(m,n):= 2m^2 + 6mn + 6n^2$
\begin{align}
\nonumber \sum_{\pi\in\mathcal{D}_{2,3N}} q^{|\pi|} & = 
\sum_{m,n\geq 0} q^{Q(m,n)+m} {3(N-2n-m)+1 \brack m }_q {2(N-2n-m)+n+1\brack n}_{q^3}\\
\nonumber &+\sum_{m,n\geq 0} q^{Q(m,n)+4m+6n+1} {3(N-2n-m)-2 \brack m }_q {2(N-2n-m)+n-1\brack n}_{q^3}\\
\nonumber &= \sum_{l=0}^N q^{3{N-2l \choose 2}} {N+1 \brack 2l+1}_{q^3}(-q,-q^5;q^6)_l 
+ q^{3N-2}\sum_{l=0}^N {q^{3{N-2l-1\choose 2}}}{N \brack 2l+1}_{q^3}(-q,-q^5;q^6)_l,\\[-1.5ex]\nonumber\\
\nonumber\sum_{\pi\in\mathcal{D}_{2,3N+2}} q^{|\pi|} &= \sum_{m,n\geq 0} q^{Q(m,n)} {3(N-2n-m+1) \brack m }_q {2(N-2n-m)+n\brack n}_{q^3}\\
&+\sum_{m,n\geq 0} q^{Q(m,n)+4m+6n+1} {3(N-2n-m) \brack m }_q {2(N-2n-m)+n\brack n}_{q^3}\\
\nonumber&= \sum_{l=0}^N q^{3{N-2l \choose 2}} {N+1 \brack 2l+1}_{q^3}(-q;q^6)_{l+1}(-q^5;q^6)_l\\
\nonumber&+ q^{3N+2}\sum_{l=0}^N {q^{3{N-2l-1\choose 2}}}{N \brack 2l+1}_{q^3}(-q;q^6)_{l+1}(-q^5;q^6)_l.
\end{align}
\end{theorem}

\section{Direct combinatorial interpretations of the double sums}\label{Sec_Motions}

We follow Kur\c{s}ung\"oz's ideas \cite{Kagan,Kagan2} and start with the partition (written in ascending order)
\begin{equation}\label{min_config}\pi=(\underbrace{2,4},\underbrace{8,10},\dots, \underbrace{3(2n-1)-1,3(2n-1)+1},3(2n-1)+1+4,3(2n-1)+1+8,\dots,3(2n-1)+1+4m)\end{equation}
from $\mathcal{D}_{1,3N+1}$, for some $3N \geq 6n+4m-3$. We note that this partition has the size $Q(m,n)$ defined in \eqref{Qnm}. Observe that $\pi$ is the partition that satisfies the Capparelli difference conditions with smallest part exactly 2, which has $n$ pairs of consecutive parts (\textit{pairs}, indicated with underbraces) with gap between these parts being exactly 2, followed by $m$ parts (\textit{singletons}) with gaps between these parts exactly 4. We call such partition a \textit{minimal configuration}. 

Given any Capparelli partition, we can always identify pairs and singletons. From the smallest part to the largest we pair up consecutive parts of the partition with gap $\leq 3$ as pairs, where a part is exclusive to a single pair, and the rest of the parts are singletons. We illustrate this with an example: 
\[\lambda = (4,\underbrace{9,12},15,20,\underbrace{24,27}).\] The partition $\lambda$ has two pairs and three singletons.

We now describe rules of \textit{motion} for the parts of $\pi$ in the style of Kur\c{s}ung\"oz \cite{Kagan, Kagan2}, which would convert $\pi$ into another Capparelli partition with the smallest part $\geq 2$ and with $n$ pairs and $m$ singletons, just as in \eqref{min_config}. These rules of motion are bijective and, in principal, can be used in reverse to transform any Capparelli partition into a unique minimal configuration.
\begin{enumerate}[i.]
\item {The motion of the largest $m$ parts that are 4 distant each:}\vspace{.1cm}

We can convert $\pi$ into \[(\underbrace{2,4},\dots,\underbrace{3(2n-1)-1,3(2n-1)+1},3(2n-1)+1+4+x_m,\dots,3(2n-1)+1+4m+x_1),\] where \[3N-6n-4m+3\geq x_1\geq x_{2}\geq \dots x_{m-1}\geq x_m\geq 0.\]
Notice that $x= (x_1,x_2,\dots,x_m)$ itself is a partition into at most $m$ parts where each part is $\leq 3N-6n-4m+3$. Also notice that this addition is bijective and can be reversed. The generating function of all such partitions $x$ is \begin{equation}\label{singleton_movement_GF}\sum_{3N-6n-4m+3\geq x_1\geq \dots\geq x_m\geq 0}q^{x_1+x_2+\dots+x_m}={m+3N-6n-4m+3\brack m}_q.\end{equation}

\item {The motion of the largest $n$ pairs that are 2 distant each:}\vspace{.1cm}

We can order the pairs $\underbrace{3k-1,3k+1}$ by their larger value $3k+1$. We move the pairs starting in order from the largest pair to the smallest pair: $\underbrace{2,4}$. Just as in the $m$ single 4-distant parts case, the amount the largest pair moves is greater or equal than the amount the second largest pair can move, which is greater or equal to the amount the third largest pair can move etc. Therefore, we do not consider or see a situation where the order of the pairs change. On the other hand, pairs might cross over singletons. Therefore, for a pair $\underbrace{x,y}$, singletons $\geq y$ play an important role. These singletons should be taken in to account when we consider the motion of such a pair.  

For some positive $k$, the pair $\underbrace{3k-1,3k+1}$ has its \textit{center} (arithmetic mean) at $3k$. If the gap between the part $3k+1$ to the next larger part (if it exists) is $\geq 5$ then we move the pair as:
\begin{equation}\label{move1}\underbrace{3k-1,3k+1}\mapsto \underbrace{3k,3k+3}.\end{equation} This motion also adds 3 to this pair's total size, and moves the center by 3/2. If the pair $\underbrace{3k,3k+3}$ is to move and if the gap between the part $3k+3$ to the next larger part (if it exists) is $\geq 5$ 
\begin{equation}\label{move2}\underbrace{3k,3k+3}\mapsto \underbrace{3(k+1)-1,3(k+1)+1}.\end{equation}
This motion once again adds 3 to this pairs total size, and moves the center by 3/2. 

It is possible that a pair can cross over a single part. There are only three possible situations of the sort. We write these situations and their outcomes here:
\begin{align}
\label{jump1}\underbrace{3k-1,3k+1},\ 3(k+2)-1 &\mapsto 3k-1,\ \underbrace{3(k+1), 3(k+2)},\\
\label{jump2}\underbrace{3k, 3(k+1)},\ 3(k+2)&\mapsto 3k,\ \underbrace{3(k+2)-1,3(k+2)+1},\\
\label{jump3}\underbrace{3k, 3(k+1)},\ 3(k+2)+1&\mapsto 3k+1,\ \underbrace{3(k+2)-1,3(k+2)+1}.
\end{align}
These motions are bijective and can be reversed. In all three of these cases the center of the pair moves $3+3/2$, even though the total size increases by 3, as in the free motion of these pairs. 

Now we can express the motion of the pairs and the generating function of partitions for such motion. The largest pair $\underbrace{3(2n-1)-1,3(2n-1)+1}$ can move up to $\underbrace{3N-1,3N+1}$, equivalently the center of the largest pair $6n-3$ can move up to $3N$. To move to the $3N$, the pair needs to cross over $m$ single parts of the partition and each move shifts the center with an extra $3$ steps. Each move is 3/2 steps forward. Hence, the largest pair needs to move \[(3N-(6n-3)-3m)\times 2/3 = 2(N-2n-m+1) \] times to reach $3N$.

There are $n$ pairs in the minimal configuration $\pi$, and all of these pairs can move at most $2(N-2n-m+1)$-many times, without crossing each other. Therefore these movements can be expressed as a partition into at most $n$ parts where each part is $\leq 2(N-2n-m+1)$. Moreover, each motion increases the total size by 3. Hence the generating function of such motions is \begin{equation}\label{pair_movement_GF}{2(N-2n-m+1)+n\brack n}_{q^3}.\end{equation}
\end{enumerate}

We illustrate the rules of motion in reverse by starting from the Capparelli partition (written in ascending order) \[\lambda = (4,\underbrace{9,12},15,20,\underbrace{24,27})\] and go to its unique minimal configuration partition. In our motions, the pairs move after the singletons. Among these pairs the smallest pair moves last. We move the smallest pair $\underbrace{9,12}$ by the reverse of the motions \eqref{move1}, \eqref{jump3}, and \eqref{move1} to get  \[\lambda' = (\underbrace{2,4},10,15,20,\underbrace{24,27}).\] The next pair in the partition $\lambda'$ is $\underbrace{24,27}$ and we can go backwards by the inverses of the moves \eqref{jump1}, \eqref{jump2}, \eqref{move1}, \eqref{jump3}, and once again \eqref{move1}. This gives us the partition \[\lambda'' = (\underbrace{2,4},\underbrace{8,10},16,21,26).\] Finally, we move these single standing parts to the smallest position they can go (which has gaps of 4 on both sides) and get the minimal configuration \[(\underbrace{2,4},\underbrace{8,10},14,18,22).\]

Putting the above arguments together, we get the following result.
\begin{theorem} The expression
\[q^{2n^2+6mn+6m^2}{3(N-2n-m+1)\brack m}_q{2(N-2n-m+1)+n\brack n}_{q^3}\] is the generating function of Capparelli partitions from $D_{1,3N+1}$ with $n$ pairs and $m$ singletons.
\end{theorem} Summing over all $m$ and $n\geq 0 $ we get the combinatorial interpretation of the left-hand side of \eqref{Fin_Cap_1_1}. \begin{corollary} \label{Corollary_G1}
\begin{equation}\label{Corollary_G1_EQN}
\sum_{m,n\geq 0} q^{2n^2+6mn+6m^2}{3(N-2n-m+1)\brack m}_q{2(N-2n-m+1)+n\brack n}_{q^3} = G_{1,3N+1}(1,1,q),\end{equation}
where $G_{1,N}$ is as defined in Section~\ref{Sec_recurrences}.
\end{corollary} 

Similarly, we can start with the minimal configurations \begin{equation}\label{pi1}\pi_1^* = (\underbrace{3,6},\underbrace{9,12},\dots,\underbrace{3(2n-1),3(2n-1)+3}, 3(2n-1)+3+4,\dots, 3(2n-1)+3+4m),\end{equation} and \begin{equation}\label{pi2}\pi_2^* = (1,\underbrace{5,7},\underbrace{11,13},\dots,\underbrace{6n-1,6n+1}, 6n+1+4,\dots, 6N+1+4m),\end{equation} and prove the following theorem with the use of Kur\c{s}ung\"oz style motions \eqref{move1}, \eqref{move2}, \eqref{jump1}-\eqref{jump3}, and the generating function relations \eqref{singleton_movement_GF} and \eqref{pair_movement_GF}.
\begin{theorem} The expression
\[ q^{Q(m,n)+m+3n}{3(N-2n-m)+2\brack m}_q {2(N-2n-m)+n+1\brack n}_{q^3}\] is the generating function of Capparelli partitions from $D_{2,3N+1}$ with $n$ pairs and $m$ singletons, where all parts are $\geq 3$. Also, \[q^{Q(m,n)+3m+6n+1}{3(N-2n-m)\brack m}_q {2(N-2n-m)+n\brack n}_{q^3}\] is the generating function of Capparelli partitions from $D_{2,3N+1}$ with $n$ pairs and $m+1$ singletons, where $1$ is always a part of the counted partitions. In this case, 1 is the singleton that does not move.
\end{theorem}
Hence, this yields the combinatorial explanation of the left-hand side of \eqref{Fin_Cap_2_1} directly.
\begin{corollary} \label{Corollary_G2}
\begin{align*}\sum_{m,n\geq 0}& q^{Q(m,n)+m+3n}{3(N-2n-m)+2\brack m}_q {2(N-2n-m)+n+1\brack n}_{q^3}\\&+\sum_{m,n\geq 0} q^{Q(m,n)+3m+6n+1}{3(N-2n-m)\brack m}_q {2(N-2n-m)+n\brack n}_{q^3} = G_{2,3N+1}(1,1,q),\end{align*}where $G_{2,N}$ is as defined in Section~\ref{Sec_recurrences}.
\end{corollary}

The interpretation of the left-hand side of \eqref{Fin_Cap_3_1} requires us to define new minimal configurations. Let \begin{equation}\label{pi1_hat}\hat{\pi}_1 := (\underline{1,5,7,11,\dots, 3(2n-1)-2,3(2n-1)+2}, 3(2n-1)+2+4, \dots, 3(2n-1)+2+4m),\end{equation} with size $Q(m,n)+m$. We call the sequence \[\underline{1,5,7,11,\dots, 6n-5,6n-1}\] the $n$-length \textit{initial chain} of $\hat{\pi}_1$. The initial chain of $\hat{\pi}_1$ consists of the first $n$ positive $\pm 1$ modulo 6 parts, and the next closest part to the initial chain is at least $4$ distant. We call $\hat{\pi}$ the minimal configuration with $n$-length initial chain and $m$ singletons.

The movement related to the initial chain is defined as the following \begin{equation}\label{long_dist_move}\underline{1,5,7,11,\dots, 6n-5,6n-1}\mapsto\underline{1,5,7,11,\dots, 6(n-1)-5,6(n-1)-1}, \underbrace{6n-3,6n},\end{equation} for any $n\in\Z_{>0}.$ This motion is bijective, creates an ordinary pair while shortening the initial chain by one. Moreover, just like \eqref{move1} and \eqref{move2}, the midpoint of $6n-5,6n-1$ moves $3/2$ units forward, and also adds $3$ to the partition overall size. After the initial movement \eqref{long_dist_move}, the created pairs move with the usual moves \eqref{move1} and \eqref{move2}. Furthermore, the end of the initial chain $\underline{1,\dots,6n-5,6n-1}$ is $4$ distant from the closest singleton $3(2n-1)+2+4$ (if it is a part of the partition) and therefore, it can move once with \eqref{long_dist_move} freely before being required to cross over the singleton. Therefore, we do not need to define (or modify) any new crossing over rules and can directly use \eqref{jump1}-\eqref{jump3}. 

Similar to \eqref{singleton_movement_GF}, one can easily show that the motions of the singletons of $\hat{\pi}$ as well as the initial chain and the later created pairs are related with the $q$-binomial coefficient. This proves the following:
% and the motion of the initial chain defined by \eqref{long_dist_move}, and later \eqref{move1}, \eqref{move2}, 
%\eqref{jump1}-\eqref{jump3} is related with the $q$-binomial \[{3(N-2n-m)+2\brack m}_q, \text{  and  } {2(N-2n-m)+n+2\brack n}_{q^3},\] %respectvely, by %analogous reasoning used in \eqref{pair_movement_GF}. Then, we have the following.  
\begin{theorem}\label{KR3_THM1} The expression
\begin{equation}\label{KR2_Summand_1}q^{Q(m,n)+m}{3(N-2n-m)+2\brack m}_q {2(N-2n-m)+n+2\brack n}_{q^3}\end{equation} is the generating function for the number of partitions from $D_{2,3N+1},$ where there are $m$ singletons and the length of the initial chain plus the number of pairs is $n$.
\end{theorem} 
To illustrate, we write 4 partitions (in ascending order) of $60$ from $D_{2,34}$, \[(\underbrace{1, 5, 7, 11}, 34),\ (\underbrace{5, 7}, \underbrace{12, 15}, 21),\ (1, 5, 7, 12, 15, 20),\text{  and  } (1,25,34).\] The first two partitions are both counted by \eqref{KR2_Summand_1} with $N=11$, $(m,n) = (1,2)$, but the latter two partitions cannot be counted by \eqref{KR2_Summand_1} for any choice of $N,\ m$, and $n$. This is due to there not being a valid initial chain configuration in these cases. 
%As an example, if we assume that $\lambda$ has the initial chain $\underbrace{1,5}$ then the presumed singleton $7$  is too close to the 
%initial chain and violates the definition of an initial chain. 
This brings us to the next minimal configuration. Let \begin{equation}\label{pi2_hat}\hat{\pi}_2 := (1,\underbrace{5,7},\underbrace{11,13},\dots, \underbrace{6n-1,6n+1}, 6n+2+4, \dots, 6n+2+4m),\end{equation} with size $Q(m,n)+4m+6n+1$. Similar to the previous case the largest pair $\underbrace{6n-1,6n+1}$ can move once freely with \eqref{move2} before the consideration of crossing over any singleton, except for the case $6n+1 = 3N+1$. Ignoring this case at the moment, the generating function of partitions from $D_{2,3N+1}$ with $n$ pairs (and an empty initial chain) and $m+1$ singletons is \begin{equation}\label{KR3_Summand2}q^{Q(m,n)+4m+6n+1}{3(N-2n-m)-1\brack m}_q {2(N-2n-m)+n\brack n}_{q^3}.\end{equation} The overlooked case appears when $N$ is even and $(m,n)=(0,N/2)$. This is due to the assumed one free movement of the largest pair in the construction of previous cases, where in this case the largest pair cannot move. We can add the related correction term to our calculations and all together we get the following theorem.
\begin{theorem}\label{KR3_THM2} The expression
\[q^{Q(m,n)+4m+6n+1}{3(N-2n-m)-1\brack m}_q {2(N-2n-m)+n\brack n}_{q^3} + \chi(N)\delta_{m,0}\delta_{n,N/2}q^{Q(0,N/2)+3N+1}\]
is the generating function for the number of partitions from $D_{2,3N+1}$ with $n$ pairs (and no valid initial chain) and $m+1$ singletons, where 1 is the sole singleton that does not move, where $\chi(N)$ is defined as in Theorem~\ref{Intro_Main_THM}.
\end{theorem} 
%The generating function in this case covers every partition from $D_{2,3N+1}$ with the smallest part equal to 1 and the second smallest %part $>5$ or the partitions with the smallest part $1$ with only ordinary pairs. 

Finally, for any given Capparelli partition from $D_{2,3N+1}$ one can uniquely identify either if the partition is coming from the minimal configuration $\hat{\pi}_1$ or $\hat{\pi}_2$ by looking for a non-empty initial chain. If the initial chain exists or if the smallest part is $>1$ then the partition is a descendant of $\hat{\pi}_1$, else it is a descendant of $\hat{\pi}_2$. Hence, putting together Theorem~\ref{KR3_THM1} and \ref{KR3_THM2} and summing over all $m,n\geq 0$ we get:
\begin{corollary} We have,
\begin{align*}
&\sum_{m,n\geq 0} q^{Q(m,n)+m}{3(N-2n-m)+2\brack m}_q {2(N-2n-m)+n+2\brack n}_{q^3}\\ &\hspace{1cm}+ \sum_{m,n\geq 0} q^{Q(m,n)+4m+6n+1}{3(N-2n-m)-1\brack m}_q {2(N-2n-m)+n\brack n}_{q^3} \\
&\hspace{1cm}+ q^{Q(0,N/2) + 3N + 1}\chi(N)\\
&=G_{2,3N+1}(1,1,q),
\end{align*}
where $\chi(N)$ is as it is defined in Theorem~\ref{Intro_Main_THM}, and $G_{2,3N+1}$ is as defined in Section~\ref{Sec_recurrences}.
\end{corollary}

\section{$q$-Trinomial Identities}\label{Sec_Trinomials}

In \cite{Andrews_Baxter}, Andrews and Baxter defined the very fruitful notion of trinomial coefficients  \begin{equation}\label{tri_def}
{N ;b;q \choose a}_2 := \sum_{j=0}^N q^{j(j+b)} {N \brack j}_q{N-j\brack j+a}_q.
\end{equation} We will drop the top variable $b$ in our notation in cases when $a=b$, and denote these cases simply as \[Tr{N \brack a}_q := {N ;a;q \choose a}_2. \]

Many recurrences of trinomial coefficients are studied and  presented in \cite{Andrews_Capparelli, Andrews_Baxter}. One of these recurrences is Lemma~2 of \cite{Andrews_Capparelli}: \begin{align}
\nonumber Tr{N+1\brack a}_q &= (1+q^N)Tr{N\brack a}_q + q^{N+a+1}Tr{N-1\brack a+2}_q\\ 
\label{Lemma_2}&+q^{N-a+1}Tr{N-1 \brack a-2}_q +q^{2N}Tr{N-1 \brack a}_q\\
\nonumber&+q^{2N-1}(1-q^{N-1})Tr{N-2 \brack a}_q,
\end{align} for any $n\in\Z_{>0}$. Note that in this recurrence the parity of the bottom variables stays the same.

Andrews' original proof of the Capparelli's identity of Theorem~\ref{capparelliconj} relies on proving the following theorem \cite{Andrews_Capparelli}.
\begin{theorem}[Andrews, 1994]\label{G1_tri_THM}
\begin{equation}\label{G1_tri}G_{1,3N+1}(1/t,t,q) = \sum_{j=-\infty}^{\infty} t^j q^{3j^2 +j} Tr{ N+1 \brack 2j}_{q^3}. \end{equation}\end{theorem}
This connection is proven by employing \eqref{Lemma_2} to show that the right side of \eqref{G1_tri} satisfies \eqref{Andrews_recurrence} with the same initial conditions.

One crucial combinatorial corollary of this result (Corollary~3 in \cite{Andrews_Capparelli}) follows immediately by extracting the powers of $t$ in \eqref{G1_tri}.
\begin{corollary}[Alladi, Andrews, Gordon, 1995]\label{Comb_Corollary_Andrews} The polynomial
\[q^{3j^2 +j} Tr{ N+1 \brack 2j}_{q^3}\]
is the generating function for the number of partitions into parts $\leq 3N+1$, in which no 1 appears as a part, the consecutive parts differ by at least 4 except if the consecutive part pairs are $(3n,3n+3)$, or $(3n-1,3n+1)$ for some $n\in\Z_{>0}$, and $j$ is equal to the number of parts congruent to 1 (mod 3) minus the number of parts congruent to $2$ (mod 3).
\end{corollary}

Furthermore, using the explicit expressions of Corollary~\ref{polyAAG} for $G_{1,3N+1}(1/t,t,q)$ and \eqref{G1_tri} together Alladi, Andrews and Gordon proved a new identity for the trinomial coefficients that appear in \eqref{G1_tri}. They used the Finite Jacobi Triple Product identity \cite[p.49, Ex.1]{Theory_of_Partitions} \[(-q^2/t , -tq^4;q^6)_l = \sum_{n=-l}^{l} t^n q^{3n^2+n} {2l\brack l-n}_{q^6} \] in \eqref{polyAAG} and extracted exponents of $t$. This yields the identity \begin{equation}\label{tri_identity_2n}Tr{N \brack 2j}_q = \sum_{n=0}^{\lfloor N/2 \rfloor} {N\brack 2n}_q {2n \brack n+j}_{q^2}q^{N-2n \choose 2}\end{equation} after $N\mapsto N-1$ and $q^3\mapsto q$.
We would like to note that they also observed the analogous identity by extensive computation \begin{equation}\label{tri_identity_2n+1}Tr{N \brack 2j+1}_q = \sum_{n=0}^{\lfloor N/2 \rfloor} {N\brack 2n+1}_q {2n+1 \brack n+j+1}_{q^2}q^{N-2n-1 \choose 2}.\end{equation} Now, with the new expression \eqref{U_poly} and the recurrence \eqref{Lemma_2} we can prove this identity \eqref{tri_identity_2n+1}.

Indeed, we would like to follow the line of Andrews from now on and present and prove the analogous results to the ones presented above. We begin with a second representation of $G_{2,3N+1}(t,1/t,q)$ using trinomial coefficients.
\begin{theorem}\label{G2_tri_exp}
\begin{equation}\label{G2_tri}G_{2,3N+1}(t,1/t,q) = \sum_{j=-\infty}^{\infty} t^j q^{3j^2 +2j} Tr{ N+1 \brack 2j+1}_{q^3}.\end{equation}
\end{theorem}

\begin{proof} Let \[C'_{N}(t,q) = \sum_{j=-\infty}^{\infty} t^j q^{3j^2 +2j} Tr{ N+1 \brack 2j+1}_{q^3}.\] Then, after writing \eqref{Lemma_2} for the trinomial \[Tr{N+1 \brack 2j+1},\] multiplying both sides with $t^j q^{3j^2 +2j}$ and summing over $j$ from $-N-1$ to $N+1$ one sees that \begin{align}
\nonumber C'_N(t,q) = \sum_{j=-\infty}^{\infty} &\left\{(1+q^{3N})Tr{N\brack 2j+1}_{q^3} + q^{3(N+2j+2)}Tr{N-1 \brack 2j+3}{q^3}\right.\\
\nonumber &+q^{3(N-2j)}Tr{N-1 \brack 2j-1}_{q^3} + q^{6N}Tr{N-1\brack 2j+1}_{q^3}\\
\nonumber &\left.+q^{6N-3}(1-q^{3N-3})Tr{N-2\brack 2j+1}_{q^3}\right\}\\[-1.5ex]\nonumber \\
\label{C_rec} & = (1+q^{3N})C'_{N-1}(t,q)+ \left(\frac{q^{3N+1}}{t} + q^{6N} + tq^{3N-1}\right) C'_{N-2}(t,q)\\
\nonumber &+ q^{6N-3}(1-q^{3N-3})C'_{N-3}(t,q).
\end{align}\vspace{.5cm}
This proves that $G_{2,3N+1}(1/t,t,q)$ and $C'(t,q)$ satisfy the same recurrence. The initial conditions \[C'_{-1}(t,q) = 0,\text{ and } C'_0(t,q) = 1\] define the sequence of $C'_N(t,q)$ uniquely, for any non-negative $N$. Comparison of these initial conditions with \eqref{Initial_conditions_Min}, where $(a,b) = (t,1/t)$ finishes the proof.
\end{proof}

Theorem~\ref{G2_tri_exp} yields the following perfect companion to Corollary~\ref{Comb_Corollary_Andrews}.

\begin{theorem}\label{Comb_Corollary2} The polynomial
\[q^{3j^2 +2j} Tr{ N+1 \brack 2j+1}_{q^3}\]
is the generating function for the number of partitions into parts $\leq 3N+1$, in which no 2 appear as a part, the consecutive parts differ by at least 4 except if the consecutive part pairs are $(3n,3n+3)$, or $(3n-1,3n+1)$ for some $n\in\Z_{>0}$, and $j$ is equal to the number of parts congruent to 2 (mod 3) minus the number of parts congruent to 1 (mod 3).
\end{theorem}  

One example of this corollary is presented in Table~\ref{Table_Corollary}.

\begin{table}[htp] \caption{Example of Corollary~\ref{Comb_Corollary2} with $N=6$ and $j=2$.}\label{Table_Corollary}

\begin{align*}q^{3\cdot 2^2 +2\cdot 2} Tr{ 6+1 \brack 2\cdot2+1}_{q^3} &=\\ q^{16} +q^{19} + &2q^{22}+2q^{25}+3q^{28}+3q^{31}+4q^{34}+3q^{37}+3q^{40}+2q^{43}+2q^{46}+q^{49}+q^{52}
%{q}^{52}+{q}^{49}+&2{q}^{46}+2{q}^{43}+3{q}^{40}+3{q}^{37}+4{q}^{34}+3\,{q}^{31}+3{q}^{28}+2{q}^{25}+2\,{q}^{22}+{q}^{19}+{q}^{16}
\end{align*}
\vspace{.1cm}
List of the related sizes and partitions, where 1 and 2 modulo 3 parts are colored differently.
\vspace{.1cm}
\[\begin{array}{cccccccc}
n 	& \pi 			& n 	& \pi 				& n 	& \pi				& n 	& \pi			\\
16  & ({\color{blue} 11},{\color{blue} 5})	& 28	& ({\color{blue} 17},{\color{blue} 8},3)		&34		&({\color{blue} 17},12,{\color{blue} 5})		& 40	&(18,{\color{blue} 14},{\color{blue} 8})		\\
19  & ({\color{blue} 14}, {\color{blue} 5})	&		& ({\color{blue} 17},{\color{blue} 11})		&		&(18,{\color{blue} 11},{\color{blue} 5})		& 43	&({\color{blue} 17},12,9,{\color{blue} 5})	\\
22  & ({\color{blue} 14},{\color{blue} 8})	& 31	& (15,{\color{blue} 11},{\color{blue} 5})	&	37		&({\color{blue} 17},{\color{blue} 11},6,3)	&		&(18,{\color{blue} 14},{\color{blue} 8},3)	\\
    & ({\color{blue} 17}, {\color{blue} 5})	&  		& ({\color{blue} 17},9,{\color{blue} 5})		&		&({\color{blue} 17},12,{\color{blue} 8})		&46 	& ({\color{blue} 17},{\color{red} 13},{\color{blue} 11},{\color{blue} 5})	\\
25	& ({\color{blue} 14},{\color{blue} 8},3)	&		& ({\color{blue} 17},{\color{blue} 11},3)		&		&(18,{\color{blue} 14},{\color{blue} 5})		&		&(18,{\color{blue} 14},9,{\color{blue} 5})	\\
	& ({\color{blue} 17},{\color{blue} 8})	& 34 	&({\color{blue} 17},{\color{blue} 11},{\color{blue} 5},1)	& 40	&({{\color{blue} 17}},{{\color{blue} 11}},{{\color{red} 7}},{{\color{blue} 5}})	& 49	&(18,15,{\color{blue} 11},{\color{blue} 5})	\\
28  & ({\color{blue} 14},9,{\color{blue} 5})	&		&({\color{blue} 17},{\color{blue} 11},6)		&		&({{\color{blue} 17}},12,{{\color{blue} 8}},3)	& 52	& ({\color{red} 19},{\color{blue} 17},{\color{blue} 11},{\color{blue} 5})
\end{array}\]
\end{table}

We now employ the Finite Jacobi Triple Product identity \cite[p.49, Ex.1]{Theory_of_Partitions} \[(-q/t ; q^6)_{l+1}(-tq^5;q^6)_l = \sum_{n=-l-1}^{l} t^n q^{3n^2+2n} {2l+1\brack l-n}_{q^6} \] in \eqref{U_poly}. Extracting exponents of $t$ in the outcome proves identity \eqref{tri_identity_2n+1}. We present this result once again as a theorem.

\begin{theorem} For any $N\in\Z_{\geq 0}$,
\begin{equation*}Tr{N \brack 2j+1}_q = \sum_{n=0}^{\lfloor N/2 \rfloor} {N\brack 2n+1}_q {2n+1 \brack n+j+1}_{q^2}q^{N-2n-1 \choose 2}.\end{equation*}
\end{theorem}

\section{Dual identities and their implications}\label{Sec_Futher}

If we combine \eqref{Corollary_G1_EQN} and \eqref{G1_tri} with $t=1$, we obtain 

\begin{equation}
\sum_{m,n\geq 0} q^{2n^2+6mn+6m^2}{3(N-2n-m+1)\brack m}_q{2(N-2n-m+1)+n\brack n}_{q^3} 
= \sum_{j=-\infty}^{\infty} q^{3j^2 +j} Tr{ N+1 \brack 2j}_{q^3}.
\label{Larissa}
\end{equation}

Next, we replace $N\mapsto N-1$, $q\mapsto \frac{1}{q}$ in \eqref{Larissa}, and multiply both sides by $q^{\frac{3N^2}{2}}$. After introducing new variable $M$ defined by the equation \begin{equation}\label{change_of_var} n = \frac{N-M-m}{2},\end{equation} doing the necessary simplifications, we get with the aid of \eqref{Binom_1_over_q}
\begin{equation}
\label{Last_formulas1_raw}
\sum_{\substack{M,m\geq 0\\ N+M+m\equiv 0\text{ (mod }2)}} q^{\frac{3M^2+m^2}{2}} {3M\brack m}_q {\frac{N+3M-m}{2}\brack 2M}_{q^3}  
= \sum_{j= -\infty}^\infty q^{3j^2+j} T_0(N,2j,q^3), 
\end{equation}
where
\begin{equation}\label{T0_def} T_0(N,a,q):= q^{\frac{N^2-a^2}{2}} Tr{N\brack a}_{q^{-1}}.
\end{equation}
We can apply the trinomial analogue of Bailey's lemma \cite{Andrews_Berkovich, Berkovich_McCoy_Pearce} to \eqref{Last_formulas1_raw}. Specifically we would like to employ \cite[$(2.10)$]{Berkovich_McCoy_Pearce} with $L\mapsto N$. Then the trinomial Bailey's lemma  \cite[$(2.10)$]{Berkovich_McCoy_Pearce} and the Jacobi Triple Product identity \cite[Thm 2.8]{Theory_of_Partitions} yield the following theorem:
\begin{theorem}
\[\sum_{\substack{N,M,m\geq 0\\ N+M+m\equiv 0\text{ (mod }2)}} \frac{q^{\frac{3N^2+3M^2+m^2}{2}}}{(q^3;q^3)_N} {3M\brack m}_q {\frac{N+3M-m}{2}\brack 2M}_{q^3}  = \frac{(-q^8,-q^{10},q^{18};q^{18})_\infty}{(q^3;q^3)_\infty}.\]
\end{theorem}

Next, we combine \eqref{G2_tri} with $t=1$ and Corollary~\ref{Corollary_G2} to obtain
\begin{align}\label{Valentine}
\nonumber
\sum_{m,n\geq 0}& q^{Q(m,n)+m+3n}{3(N-2n-m)+2\brack m}_q {2(N-2n-m)+n+1\brack n}_{q^3}\\&+
\sum_{m,n\geq 0} q^{Q(m,n)+3m+6n+1}{3(N-2n-m)\brack m}_q {2(N-2n-m)+n\brack n}_{q^3} \\\nonumber & \hspace{-1cm} =
\sum_{j=-\infty}^{\infty} q^{3j^2 +2j} Tr{ N+1 \brack 2j+1}_{q^3}.
\end{align}
As before, we replace $N\mapsto N-1$, $q\mapsto \frac{1}{q}$ in \eqref{Valentine}, and multiply both sides by $q^{\frac{3N^2}{2}}$
to derive
\begin{align}
\nonumber
\sum_{\substack{M,m\geq 0\\ N+M+m\equiv 1\text{ (mod }2)}}& q^{\frac{3M^2+6M+m^2}{2}} {3M\brack m}_q {\frac{N+3M-m-1}{2}\brack 2M}_{q^3} \\
\label{Last_formulas2_raw}&+\sum_{\substack{M,m\geq 0\\ N+M+m\equiv 1\text{ (mod }2)}} q^{\frac{3M^2+6M+m^2}{2}+1} {3M+2\brack m}_q 
{\frac{N+3M-m+1}{2}\brack 2M+1}_{q^3} 
 \\\nonumber&\hspace{-2cm}= \sum_{j= -\infty}^\infty q^{3j^2+2j} T_0(N, 2j+1,q^{3}).
\end{align}
Dividing both sides of \eqref{Last_formulas2_raw} with $(q^3;q^3)_N$, and employing $(2.10)$ in \cite{Berkovich_McCoy_Pearce} followed by the Jacobi Triple Product identity \cite[Thm 2.8]{Theory_of_Partitions} proves the following theorem:

\begin{theorem}
\begin{align}
\nonumber
\sum_{\substack{N,M,m\geq 0\\ N+M+m\equiv 1\text{ (mod }2)}}& \frac{q^{\frac{3N^2+3M^2+6M+m^2-3}{2}}}{(q^3;q^3)_N} {3M\brack m}_q 
{\frac{N+3M-m-1}{2}\brack 2M}_{q^3} \\
\label{Last_formulas2}&+\sum_{\substack{N,M,m\geq 0\\ N+M+m\equiv 1\text{ (mod }2)}} \frac{q^{\frac{3N^2+3M^2+6M+m^2-1}{2}}}{(q^3;q^3)_N} {3M+2\brack m}_q {\frac{N+3M-m+1}{2}\brack 2M+1}_{q^3} 
\\\nonumber &\hspace{-2cm}= \frac{(-q,-q^{17},q^{18};q^{18})_\infty}{(q^3;q^3)_\infty}. \end{align}
\end{theorem}

Furthermore, one can also get curious new $q$-series identities by simply tending $N\rightarrow \infty$ in the identities \eqref{Last_formulas1_raw} and \eqref{Last_formulas2_raw}. These results require the following limit first proven by Andrews--Baxter \cite{Andrews_Baxter}: \begin{equation}\label{AB_limit}
\lim_{N\rightarrow\infty} T_0(2N,a,q^2) + T_0(2N+1,a,q^2) = \frac{(-q;q^2)_\infty}{(q^2;q^2)_\infty}.
\end{equation} We take two copies of \eqref{Last_formulas1_raw} one with $N\mapsto 2N$ and one with $N\mapsto 2N+1$ and add these copies together. Same procedure is done for \eqref{Last_formulas2_raw}. After sending $N$ to $\infty$, we use the q-Binomial Theorem \cite[p. 354, II.3]{Gasper_Rahman} on the left-hand side to sum the variable $m$ out. We use the limit \eqref{AB_limit} and the Jacobi Triple Product identity \cite[Thm 2.8]{Theory_of_Partitions} on the right-hand side. After the dilation $q\mapsto q^2$ and necessary simplifications we get the following two identities.

\begin{theorem}\label{Slater_style_THM}
\begin{align}
\label{Slater_style1}\sum_{n\geq 0} \frac{q^{3n^2}(-q;q^2)_{3n}}{(q^6;q^6)_{2n}} &=  (-q^4,-q^8;q^{12})_\infty(-q^3;q^3)_\infty,\\
\label{Slater_style2}\sum_{n\geq 0} \frac{q^{3n^2+6n}(-q;q^2)_{3n}}{(q^6;q^6)_{2n}} &+\sum_{n\geq 0} \frac{q^{3n^2+6n+2}(-q;q^2)_{3n+2}}{(q^6;q^6)_{2n+1}} = (-q^2,-q^{10};q^{12})_\infty(-q^3;q^3)_\infty.
\end{align}
\end{theorem}

We would like to give a combinatorial interpretation of Theorem~\ref{Slater_style_THM}. To this end, we would like to rewrite the left-hand side sum \eqref{Slater_style1} as follows: \begin{equation}\label{Slater_left_1}\sum_{n\geq 0} \frac{q^{3n^2}}{(q^3;q^6)_n}\frac{(-q,-q^5;q^6)_{n}}{(q^{12};q^{12})_{n}}.\end{equation} We interpret the factor $3n^2$ as the generating function for the partition \[(6n-3,6n-9,\dots,9,3),\] which has $n$, all $3$ modulo 6, consecutive parts. The factor $3n^2 / (q^3;q^6)_n$ is the generating function for the partitions into  $3$ modulo 6 parts $\leq 6n-3$, where every part size appears at least once. The second factor \[\frac{(-q,-q^5;q^6)_{n}}{(q^{12};q^{12})_{n}}\] is the generating function of the partitions into distinct parts $\pm 1$ modulo 6 that are $\leq 6n-1$ and into 0 modulo 12 parts (which may repeat) that are $\leq 12n$. Therefore, all together, one can see that \eqref{Slater_left_1} is the generating function for the number of partitions with a largest $3$ modulo 6 (which might repeat) part $\lambda$, where all the smaller $3$ modulo 6 parts also appear as parts in the partition, only even parts are $0$ modulo 12 and they are each $\leq 2\lambda+6$, and distinct $\pm 1$ modulo 6 parts $\leq \lambda + 2 $. We denote the set of partitions that satisfy the above conditions by $\mathcal{C}_{1}$. 

Similar interpretation can be given to the left-hand side of \eqref{Slater_style2}. First, we rewrite it as \[(1+q^2)\sum_{n\geq 0} \frac{q^{3n^2+6n}}{(q^3;q^6)_{n+1}}\frac{(-q,-q^5;q^6)_{n}}{(q^{12};q^{12})_{n}}.\] Dividing both sides of \eqref{Slater_style2} by $(1+q^2)$ yields 
\begin{equation}
\sum_{n\geq 0} \frac{q^{3n^2+6n}}{(q^3;q^6)_{n+1}}\frac{(-q,-q^5;q^6)_{n}}{(q^{12};q^{12})_{n}}
= (-q^{10},-q^{14};q^{12})_\infty(-q^3;q^3)_\infty.
\label{Slater_style3}
\end{equation}

The sum on the left is the generating function for the partitions that satisfy the conditions of $\mathcal{C}_{1}$, except that $3$ may or may not be a part, the largest $3$ modulo 6 part $\lambda$ occurs at least once, 0 modulo 12 parts are all $\leq 2\lambda-6$, and distinct $\pm 1$ modulo 6 parts $\leq \lambda - 4 $. We denote the set of these partitions by $\mathcal{C}_{2}$.

Also, for $k=1,\ 2$, let $\mathcal{B}_{k}$ be the set of partitions into distinct parts divisible by 3 or congruent to $\pm (6 - 2k)$ modulo 12, and 2 is never a part. Then, Theorem~\ref{Slater_style_THM} has its combinatorial equivalent:

\begin{theorem}[Dual of the Capparelli's Partition Theorem]\label{Combinatorial_THM_1_1}
\[
\sum_{\pi \in  \mathcal{C}_{1}} q^{|\pi|} = \sum_{\pi\in\mathcal{B}_{1}}q^{|\pi|},\text{  and  }
\sum_{\pi \in \mathcal{C}_{2}} q^{|\pi|} =\sum_{\pi\in\mathcal{B}_{2}}q^{|\pi|}.
\]
\end{theorem}

One example of this result is given in Table~\ref{Table_Combinatorial_THM_1_1}.

\begin{table}[htp]\caption{Example of Theorem~\ref{Combinatorial_THM_1_1} for partitions with size 21} \label{Table_Combinatorial_THM_1_1}
\[\begin{array}{cc|ccc}
\pi \in \mathcal{C}_{1} & \pi\in\mathcal{B}_{1} &  \pi \in \mathcal{C}_{2} & \pi\in\mathcal{B}_{2} \\
(12,5,3,1) 		& (21)		  & (12,9)		  & 			 (21)  \\
(12,3,3,3) 		& (18,3)	  & (9,9,3)		& 			 (18,3) \\
(9,9,3)  		  & (15,6)	  & (9,5,3,3,1)	& 			 (15,6)\\
(9,5,3,3,1) 	& (12,9)	  & (9,3,3,3,3)		& 		 (12,9)\\
(9,3,3,3,3) 	& (12,6,3)	& (3,3,3,3,3,3,3)				& 						 (12,6,3)\\
(5,3,3,3,3,3,1)	& (9,8,4)	&				&						 \\
(3,3,3,3,3,3,3)	& (8,6,4,3)	&				&					\\
\end{array}\]
\end{table}

Before we move on, we remark that \eqref{Slater_style1} and \eqref{Slater_style3} are special cases $(q,z)\rightarrow(q^3,-q^2)$ of the following identities

\begin{equation}
\sum_{n\geq 0} q^{n^2}\frac{(zq,z^{-1}q;q^2)_n}{(q;q^2)_n(q^4;q^4)_n}
= (-q;q)_\infty(zq^2,\frac{q^2}{z};q^4)_\infty,
\label{FG1}
\end{equation}

\begin{equation}
\sum_{n\geq 0} q^{n^2+2n} \frac{(zq,z^{-1}q;q^2)_n}{(q;q^2)_{n+1}(q^4;q^4)_n}
= (-q;q)_\infty(zq^4,\frac{q^4}{z};q^4)_\infty,
\label{FG2}
\end{equation}
respectively. Both formulas can easily be proven using techniques of \cite{Bailey}.

We would like to give a secondary combinatorial interpretation of Theorem~\ref{Slater_style_THM}. Let $\nu(\pi)=\nu$ be the number of parts of $\pi$, and ${}_e\nu_{<m}(\pi)$ be the number of even parts of $\pi$ that are strictly less than $m$, for some $m\in\mathbb{Z}_{\geq0}$. For $M=0$ and 2, let $\mathcal{A}_{M}$ be the set of partitions $\pi = (\lambda_1,\lambda_2,\dots, \lambda_{\nu(\pi)})$, with distinct even parts, where every odd part $\geq M$ can repeat up to 3 times, and for any $i\leq \lceil \nu(\pi)/3\rceil$ \begin{enumerate}[i.]
\item $\nu(\pi) \equiv 0$ or $M$ mod 3,
\item $\lambda_{3i}-\lambda_{3i+1} \geq 2$,
\item $|\lambda_{3i-2}-2\lambda_{3i-1}+\lambda_{3i}(1-\delta_{M,2}\delta_{3i,\nu+1})|\leq 1-\delta_{M,2}\delta_{3i,\nu+1}$, 
\item $\lambda_{3i-r} \equiv\hspace{-.1cm} 1+M + 2\hspace{-.1cm}\left(\nu-i+\delta_{3\lceil \nu/3\rceil-\nu,1}(\lambda_{\nu}+\chi(\lambda_{\nu})+1)+{}_e\nu_{<\lambda_{3i-r}}(\pi)\right) + \chi(\lambda_{3i-r})$ mod 4, for $r=0$ and $2$.
\end{enumerate}
We would like to note that that these conditions slightly resemble the conditions of the Capparelli's Companion Theorem: Theorem~1.3 of \cite{BerkovichUncu1}.

\begin{theorem}[Companion to the Dual of Capparelli's Partition Theorem]\label{Combinatorial_THM_1}
\[
\sum_{\pi \in  \mathcal{A}_{0}} q^{|\pi|} = \sum_{\pi\in\mathcal{B}_{1}}q^{|\pi|},\text{  and  }
\sum_{\pi \in \mathcal{A}_{2}} q^{|\pi|}  =\sum_{\pi\in\mathcal{B}_{2}}q^{|\pi|}.
\]
\end{theorem}

We illustrate these results in Table~\ref{Table_Combinatorial_THM_1}.

\begin{table}[h]\caption{Example of Theorem~\ref{Combinatorial_THM_1} for partitions with size 21} \label{Table_Combinatorial_THM_1}
\[\begin{array}{cc|cc}
\pi \in \mathcal{A}_{0} & \pi\in\mathcal{B}_{1} &  \pi \in \mathcal{A}_{2}& \pi\in\mathcal{B}_{2} \\
(13,7,1) 		& (21)		& 	(14,7)		& 	  (21)  \\
(12,7,2) 		& (18,3)	& (11,7,3)		& (18,3) \\
(9,7,5)  		& (15,6)	& (10,7,4)		& (15,6)\\
(9,6,3,1,1,1) 	& (12,9)	& (7,7,7)		& (12,9)\\
(8,7,6) 		& (12,6,3)	& (7,6,5,2,1)	& (12,6,3)\\
(8,6,4,1,1,1)	& (9,8,4)	&				&	 \\
(5,5,5,3,2,1)	& (8,6,4,3)	&				&	\\
\end{array}\]
\end{table}

\section{Outlook}\label{Sec_Outlook}

The nine polynomial identities of Theorems~\ref{Intro_Main_THM}, \ref{THM_C1_23}, \ref{THM_C2_23}, and \ref{THM_C3_23} that imply Capparelli's identities are merely a collection of polynomial identities that imply these beautiful results. There are other polynomial identities that imply Capparelli's partition theorems. One example is the following.

\begin{theorem}\label{THM_Outlook} For $Q(m,n):= 2m^2 + 6mn + 6n^2$, and any non-negative integer $M$,
\begin{equation}\label{Eqn_Outlook}
\sum_{n= 0}^{\lfloor M/2\rfloor}\sum_{m=0}^{M-2n} \frac{q^{Q(m,n)} (q^3;q^3)_M }{(q;q)_m (q^3;q^3)_n (q^3;q^3)_{M-2n-m}}
= \sum_{j=-M}^M q^{3j^2+j} {2M \brack M-j}_{q^3}, 
\end{equation} where $\lfloor x\rfloor$ denotes the greatest integer $\leq x$.
\end{theorem}

One can directly prove this identity using recurrences. With computer assistance, one can get (and prove) the recurrences for the left-hand and right-hand sides of the above identity using Riese's implementation qMultiSum \cite{qMultiSum} and Paule and Riese's implementation qZeil \cite{qZeil}, respectively. Then, similar to Section~\ref{Sec_recurrences} and \ref{Sec_MultiSums}, one either shows that the right-hand side also satisfies the recurrences of the left-hand side, or one combines the two recurrences automatically using Kauers and Koutschan's implementation qGeneratingFunctions \cite{qGeneratingFunctions} to find a recurrence satisfied by both sides of \eqref{Eqn_Outlook}. The rest of the proof is just confirmation of the initial terms.

As $M$ tends to infinity, the left-hand side of \eqref{Eqn_Outlook} turns into the left-hand side of \eqref{KR1_GF} that appears in \eqref{Cap1}. The right-hand side of \eqref{Eqn_Outlook} requires some simplifications. Using \eqref{Binom_limit2} and the Jacobi Triple Product Identity \cite[p.49, Ex.1]{Theory_of_Partitions} we reach the right-hand side product of \eqref{Cap1}.

There are other polynomial identities similar to \eqref{Eqn_Outlook}. It should be noted that at first sight ---without the knowledge of the identity--- it is not even clear that the left-hand side polynomial of \eqref{Eqn_Outlook} has non-negative integer coefficients. 

Surprisingly, we have also found an identity for the sum of two Capparelli products:
\begin{theorem}\label{THM_Outlook2}For $Q(m,n):= 2m^2 + 6mn + 6n^2$,
\begin{equation}\label{EQN_Outlook2}
\sum_{m,n\geq 0} \frac{q^{Q(m,n)-2m-3n}}{(q;q)_m(q^3;q^3)_n} = (-q^2,-q^4;q^6)_\infty(-q^3;q^3)_\infty  + (-q,-q^5;q^6)_\infty(-q^3;q^3)_\infty .
\end{equation}
\end{theorem}
The reader is invited to compare \eqref{EQN_Outlook2} with identities \eqref{Cap1}-\eqref{Cap3}.

Finally, we would like to point out that the identity \eqref{Slater_style1} is the first member of the following new infinite hierarchy:
\begin{theorem}\label{THM_Outlook3} For positive integer $\nu$, we have 
\begin{align}
\nonumber\sum_{n_1,n_2,\dots,n_\nu\geq 0} &\frac{q^{3(N^2_1+N^2_2+\dots + N^2_\nu)} (-q;q^2)_{3n_\nu}}{(q^6;q^6)_{n_1}(q^6;q^6)_{n_2}\dots (q^6;q^6)_{n_{\nu-1}}(q^6;q^6)_{2n_\nu}}\\
\label{EQN_Outlook3} &\hspace{2cm}= \frac{(-q^3;q^3)_\infty}{(q^{12};q^{12})_\infty} (q^{6(\nu+1)},-q^{3\nu+1},-q^{3\nu+5};q^{6(\nu+1)})_\infty,
\end{align}where $N_i := n_i+ n_{i+1}+\dots +n_\nu$ for $i=1, 2, \dots, \nu$.
\end{theorem}
We prove Theorems~\ref{THM_Outlook}--\ref{THM_Outlook3} in \cite{BerkovichUncu2} and \cite{BerkovichUncu3}.

\section{Note Added}

Andrew Sills brought to our attention two formulas of F. J. Dyson \cite[$(7.5, 7.6)$]{Bailey}. These formulas can be stated as

\begin{equation}
\sum_{n\geq 0} \frac{q^{3n^2}(q^2,q^4;q^6)_n}{(q^{12};q^{12})_n(q^3;q^6)_n} = (-q^3;q^3)_\infty(q^5,q^7;q^{12})_\infty,
\label{7.5}
\end{equation}
\begin{equation}
1-\sum_{n\geq 1} \frac{q^{3n^2-2}(q^2;q^6)_{n+1}(q^4;q^6)_{n-1}}{(q^{12};q^{12})_n(q^3;q^6)_n} = (-q^3;q^3)_\infty(q,q^{11};q^{12})_\infty.
\label{7.6}
\end{equation}
He suggested that the above formulas are similar to those in Theorem~\ref{Slater_style_THM}. 

Clearly, \eqref{7.5} is a special case $(q,z)\rightarrow(q^3,q)$ of \eqref{FG1}. In addition, the equation 
\eqref{7.6} is a special case $(q,z)\rightarrow(q^3,-q^2)$ of the following new identity

\begin{equation}
1+\frac{1}{z}\sum_{n\geq 1} \frac{q^{n^2}(-z;q^2)_{n+1}(-\frac{q^2}{z};q^2)_{n-1}}{(q^{4};q^{4})_n(q;q^2)_n} 
= (-q;q)_\infty(-\frac{q}{z},-zq^{3};q^{4})_\infty,
\label{7.6a}
\end{equation} 
which can easily be proven by techniques of \cite{Bailey}.
Partition theoretical interpretation of this formula will be given elsewhere.

\section{Acknowledgements}

Authors would like to thank Veronika Pillwein, Ka\u{g}an Kur\c{s}ung\"oz, Andrew Sills, and Frank Garvan for many helpful discussions. Authors would also like to extend their gratitude to the organizers of the Combinatory Analysis 2018 conference, where this work started.

Research of the first author is partly supported by the Simons Foundation, Award ID: 308929. Research of the second author is supported by the Austrian Science Fund FWF, SFB50-07 and SFB50-09 Projects. The second author would also like to thank the Simons Foundation, which supported his visit to University of Florida.

\end{document}